\documentclass[10pt]{amsart}

\usepackage[total={6in, 8.4in}]{geometry}

\usepackage[utf8]{inputenc}

\usepackage{mathtools,xcolor,marginnote}
\usepackage{mathscinet}

\usepackage{pdfpages}
\usepackage{tikz,tikz-cd}
\usetikzlibrary{decorations.pathmorphing}

\usepackage{float}

\newtheorem*{theorem-a}{Theorem A}

\theoremstyle{plain}
\newtheorem{Theorem}{Theorem}
\newtheorem{Proposition}{Proposition}
\newtheorem{Lemma}{Lemma}
\newtheorem{Corollary}{Corollary}

\theoremstyle{remark}
\newtheorem{Remark}{Remark}

\theoremstyle{definition}
\newtheorem{Definition}{Definition}

\DeclareMathOperator{\re}{Re}
\DeclareMathOperator{\id}{id}

\DeclareMathOperator{\im}{Im}
\DeclareMathOperator{\Aff}{\mathcal{A}ff}
\DeclareMathOperator{\sgn}{sgn}

\begin{document}

\title[]{A Maximum Modulus Theorem for functions admitting {S}tokes phenomena, and specific cases of {D}ulac's Theorem
}

\author[J. Palma-M\'arquez]{Jes\'us Palma-M\'arquez}
\address{Weizmann Institute of Science\\
Rehovot\\
Israel
}
\email[]{jesus.palma@weizmann.ac.il}

\author[M. Yeung]{Melvin Yeung}
\address{Hasselt University\\
Campus Diepenbeek, Agoralaan Gebouw D\\
3590 Diepenbeek\\
Belgium
}
\email{melvin.yeung@uhasselt.be}

\thanks{The research of J. Palma-M\'arquez was partially supported by Papiit Dgapa UNAM IN103123, by the Israel Science Foundation (Grant No. 1347/23) and by funding received from the MINERVA Stiftung with the funds from the BMBF of the Federal Republic of Germany. This project has received funding from the European Research Council (ERC) under the European Union’s Horizon 2020 research and innovation program (Grant Agreement No. 802107).
The research of M. Yeung was supported by `Research Foundation -- Flanders, FWO' file number 11E6821N}
\date{\today}

\begin{abstract}
We study large classes of real-valued analytic functions that naturally emerge in the understanding of Dulac's problem, which addresses the finiteness of limit cycles in planar differential equations. Building on a Maximum Modulus-type result we got, our main statement essentially follows. Namely, for any function belonging to these classes, the following dichotomy holds: either it has isolated zeros or it coincides with the identity. As an application, we prove that the non-accumulation of limit cycles holds around a specific class of the so-called superreal polycycles.
\end{abstract}

\maketitle

\section*{Introduction}\label{sec:intro}

In this manuscript, generally speaking, we prove that concrete families of real-valued functions arising in the study of Dulac's problem about the finiteness of limit cycles in planar differential equations are well-behaved. Namely, going into the Complex Analytic domain, we prove that for any function belonging to such classes it happens that either it has isolated zeros or it coincides with the identity. In doing so, we generalize concepts first introduced by Yu.S. Ilyashenko in his seminal work \cite{ilyashenkoFiniteness}.

One of our main motivations is to gain a better understanding of Dulac's problem, because the proposed proofs; cf., \cite{ecallePreuve,ilyashenkoFiniteness}, seem to be far from fully understood by most of the specialists. For an overview on Dulac's problem and the main definitions and concepts related to, we refer to the introduction of Yu.S. Ilyashenko's book \cite{ilyashenkoFiniteness}. We also recommend the recent book \cite[\S 24]{Ilyashenko08lectureson}, which contains a historic review of the still unsolved Hilbert sixteenth problem, as well as the aforementioned Dulac's problem, including a complete proof of the non accumulation of limit cycles around the so-called hyperbolic polycycles, originally proved in \cite{Ilyashenko1984}. Throughout this manuscript, we will primarily follow the definitions and notations used in these works.

In \cite{Melvin24}, M. Yeung offered a constructive approach (in the traditional sense) to proving non oscillation of return maps of groups of polycycles, drawing on a part of the ideas present in Yu.S. Ilyashenko's work \cite{ilyashenkoFiniteness}. Essentially, for a suitable choice of classes $\mathcal{R}$ and $\mathcal{NC}$ (that we will define rigorously later), linked to the types of saddles appearing in the polycycles of interest, one comes to the point that the primary focus is on the classes of functions:
\begin{equation*}
    \langle \Aff, A^{i}\mathcal{R}^{0}, A^{j}\mathcal{NC}\mid 0 \leq i \leq n, 0 \leq j \leq n - 1\rangle,
\end{equation*}
where $\langle . \rangle$ means group generated under composition, $A$ is conjugation with the exponential; i.e., $A(f) \coloneqq \ln \circ f \circ \exp$, and $\mathcal{R}^{0}$ is the subset of $\mathcal{R}$ with identity linear part.

We will address how to show non oscillation for certain elements of the above group when the elements of the class $\mathcal{NC}$ exhibit Stokes phenomena of a type to be defined below.

In particular, by induction on $n$, we will be considering elements of the form:
\begin{equation*}
    \Aff \circ \langle \mathcal{R}, \mathcal{NC}\rangle \circ \cdots \circ A^{n}(\langle \mathcal{R}, \mathcal{NC}\rangle)  \circ \cdots \circ \langle \mathcal{R}, \mathcal{NC}\rangle \circ \Aff .
\end{equation*}
That is to say, the amount of $A$ that is put around $\langle \mathcal{R}, \mathcal{NC}\rangle$ first increases to $n$, and then decreases back to zero. For simplicity's sake we will define a class $\mathcal{NC}$ such that $\mathcal{R} \subset \mathcal{NC}$, essentially because we will at no point need the larger domain of $\mathcal{R}$ which was vital for induction in the proof of \cite{ilyashenkoFiniteness}.

The classes $\Aff$ and $\mathcal{NC}$ we will consider here actually are the following (for more detail see Definition \ref{DefredefNC}):

\begin{Definition}
We define the class $\Aff$ to be the class of affine real analytic functions with positive derivative; i.e., the functions:
\begin{equation*}
    \zeta \mapsto \alpha\zeta + \beta \quad,\quad \alpha, \beta \in \mathbb{R}, \alpha > 0.
\end{equation*}
We define $\mathcal{R}$ to be the class of almost regular functions of \cite[Definition 24.27]{Ilyashenko08lectureson}, which are real on the real axis.
\end{Definition}

\begin{Definition}
Let $\mathcal{NC}$ be the set of real analytic germs at infinity which can be extended to extendable cochains on some standard quadratic domain; i.e., a domain of the form:
\begin{equation*}
    \Omega \coloneqq \Psi(\mathbb{C}^{+}),
\end{equation*}
with $\Psi(\zeta) = \zeta + C\sqrt{\zeta + 1}$ for some $C > 0$ (positive branch of the square root) and $\mathbb{C}^{+}$ being the complex half-plane with positive real part.

With partition given by the lines $\im(\zeta) = k\pi, k \in \mathbb{Z}, k \neq 0$; i.e., on each of the strips $\Pi$ in $\mathbb{C}^{+}$ bounded by two adjacent lines of the form $\im(\zeta) = k\pi$, except $k = 0$, we get an analytic function which can be analytically continued to a strip with a larger width. We will denote by $\Pi_{\epsilon}$ the strip $\Pi$ widened by $\epsilon$ on both sides (still inside the standard quadratic domain).

Then the class $\mathcal{NC}$ is the subset of such cochains $f$ for which:
\begin{enumerate}
    \item{There exists some series:
    \begin{equation*}
        \zeta + \sum P_{n}(\zeta)e^{-c_{n}\zeta}\,,
    \end{equation*}
    with the $P_{n}$ real polynomials, $c_{n} > 0$ real and going to $+\infty$ such that for any $m > 0$ there exists a finite sum $S_{N}$ up to some $N$ which approximates all the component functions of $f$ uniformly up to accuracy $O(e^{-m\zeta})$; i.e., there exists some $C > 0, \epsilon > 0$ and some $\xi_{0} > 0$ such that for all strips $\Pi$, for all $\zeta \in \Pi_{\epsilon}$ with $\re(\zeta) > \xi_{0}$ for the analytic function $f_{\Pi}$ on the strip $\Pi$:
    \begin{equation*}
        |f_{\Pi}(\zeta) - S_{N}(\zeta)| < Ce^{-m\re(\zeta)}\,.
    \end{equation*}
    }
    \item{There exists some $\epsilon > 0$, $C, C' > 0$ such that for any two strips $\Pi, \Pi'$ with respective functions $f_{\Pi}, f_{\Pi'}$ we have for all $\zeta \in \Pi_{\epsilon} \cap \Pi'_{\epsilon}$:
    \begin{equation*}
        |f_{\Pi}(\zeta) - f_{\Pi'}(\zeta)| \leq Ce^{-C'e^{\re(\zeta)}}\,.
    \end{equation*}
    }
    \item{The function $f_{\Pi_{0}}$ on the strip $\Pi_{0}$ containing the real axis, is the original real analytic germ in the class $\mathcal{NC}$.}
\end{enumerate}
\end{Definition}

\begin{Remark}
To those familiar with the Theory of summability this may look rather familiar. In fact, the way the class $\mathcal{NC}$ comes up in \cite{ilyashenkoFiniteness} is that they contain the normalization maps of semihyperbolic saddles to their formal normal form on the centre side (suitably normalized and put in the logarithmic chart $\zeta = -\ln(z)$, with $z$ the usual coordinate).

Considering this, we emphasize that we have made a simplification assuming $f$ to be real on the real axis. In essence, this only holds for semihyperbolic saddles where the Martinet-Ramis moduli; cf., \cite{MartinetRamisSemihyperbolic} and \cite[\S 3]{IlyashenkoStokesPhenomena}, on the axis corresponding to the one-sided transversal for the Dulac map are all zero, an infinite codimension. But without this assumption asymptotics becomes much more tedious and even in the more general case it is entirely unknown what to do.
\end{Remark}

\begin{Definition}
We will call a semihyperbolic saddle with Martinet-Ramis moduli as above \emph{superreal}. The term ``superreal" was suggested by Ilyashenko.
\end{Definition}

The following is then our main Theorem:

\begin{theorem-a}\label{Theorem:Main Thm}
Let $g$ be an element of:
\begin{equation*}
    \Aff \circ \langle \mathcal{NC}\rangle \circ A(\langle \mathcal{NC}\rangle) \circ \cdots \circ A^{n - 1}(\langle\mathcal{NC}\rangle) \circ A^{n}(\langle \mathcal{NC}\rangle) \circ A^{n - 1}(\langle\mathcal{NC}\rangle) \circ  \cdots \circ A(\langle \mathcal{NC}\rangle)\circ \langle \mathcal{NC}\rangle \circ \Aff\,.
\end{equation*}
Then either $g \equiv \id$ or $g$ has no fixed points close enough to $+\infty$.

In particular if for all $\mu > 0$, for $x$ large enough real positive (depending on $\mu$):
\begin{equation*}
    |g(x) - x| < e^{-\mu \exp^{n}(x)}\,,
\end{equation*}
($n$-fold composition of the exponential) then $g \equiv \id$.
\end{theorem-a}

As a consequence of Theorem A, we obtain a positive answer for a restricted version of Dulac's problem. To do this we first need to introduce the notion of \emph{depth} of a polycycle; cf., \cite{Melvin24}.

So take a polycycle homeomorphic to a circle in a vector field on a real analytic $2$-manifold. We may now parametrize our polycycle, say $\Gamma$, with $\gamma\colon [0, 1] \to \Gamma$, starting at an arbitrary point $x \in \Gamma$, say that $x$ is not equal to an equilibrium. Suppose that $\gamma$ is injective on $(0, 1)$.

Then for a $t \in [0, 1]$ we can define the \emph{depth} of $\gamma$ at $t$, $D(\gamma, t)$ as follows:
\begin{equation*}
    D(\gamma, t) \coloneqq \mathfrak{C}-\mathfrak{H}\,,
\end{equation*}
where
\begin{equation*}
    \mathfrak{C}\coloneqq \# \{\text{semihyperbolic saddles in }\gamma((0, t]) \text{ from the centre direction}\}\,,
\end{equation*}
and
\begin{equation*}
    \mathfrak{H} \coloneqq \# \{\text{semihyperbolic saddles in }\gamma((0, t]) \text{ from the hyperbolic direction}\}\,.
\end{equation*}
Note first that this is well-defined because a polycycle only has a transit map along a `single side'.

Using this we can define a particular class of polycycles:

\begin{Definition}
We call a polycycle \emph{superreal} if every semihyperbolic saddle in it is superreal. And, we call a polycycle \emph{balanced} if $D(\gamma, 0) = D(\gamma, 1)$.

Moreover, we say that a polycycle \emph{has one turn} if there exists a parametrization $\gamma$ such that $D(\gamma, t)$ only goes from decreasing to increasing once; i.e., there exists $t_{0}\in(0,1)$ such that before $t_{0}$ only semihyperbolic saddles going to the central manifold are encountered and after $t_{0}$ only semihyperbolic saddles going away from the central manifold are encountered, hyperbolic saddles may be encountered anywhere.
\end{Definition}

Then from our Main Theorem we have the following partial positive answer to Dulac's problem:

\begin{Corollary}
Any superreal and balanced real analytic polycycle with only one turn has a (one-sided) neighbourhood without limit cycles.
\end{Corollary}
\begin{proof}[sketch]
It is known (cf., \cite[Lemma 24.40]{Ilyashenko08lectureson}) that any Dulac map of a hyperbolic saddle gives an element of the above described class $\mathcal{NC}$, even without the cochain part; i.e., it is a single function on that domain.

It is also known (essentially \cite{MartinetRamisSemihyperbolic}, but also \cite{Melvin24} for how to manipulate the formal normal form, see also \cite[p. 43]{ilyashenkoFiniteness}, where it is given a table containing all the possible different maps one has to deal with) that the Dulac map of a superreal semihyperbolic saddle can be decomposed into an analytic normalization on the hyperbolic side, something close to an exponential function and an element of the class $\mathcal{NC}$, even up to a shifted half-plane instead of a standard quadratic domain in the logarithmic chart.

The assumption that the polycycle only has one turn then puts it in the correct form to apply Theorem A to the return map of the polycycle, indeed, such a polycycle as represented in \cite{Melvin24} will have some form like this:

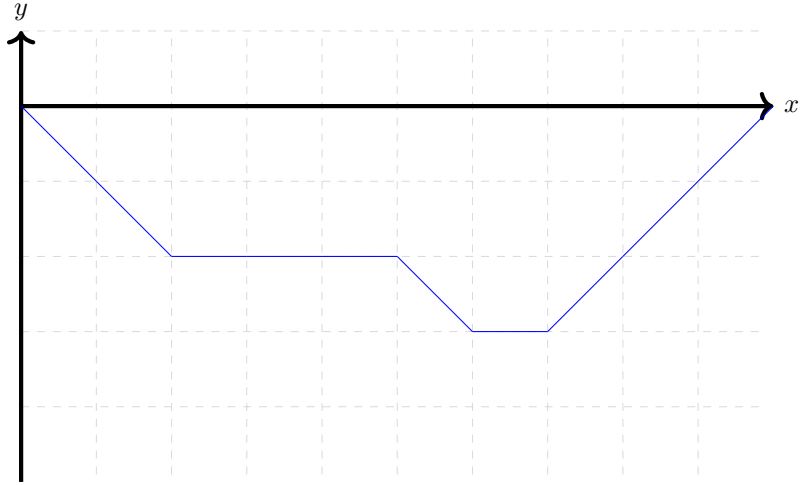
\begin{figure}[H]
	\begin{tikzpicture}
		\draw[help lines, color=gray!30, dashed] (0,-4.9) grid (9.9,1);
		\draw[->,ultra thick] (0,0)--(10,0) node[right]{$x$};
		\draw[->,ultra thick] (0,-5)--(0,1) node[above]{$y$};
		\draw[blue](0, 0)--(1, -1);
		\draw[blue](1, -1)--(2, -2);
		\draw[blue](2, -2)--(3,-2);
		\draw[blue](3, -2)--(4, -2);
		\draw[blue](4, -2)--(5, -2);
		\draw[blue](5, -2)--(6, -3);
		\draw[blue](6, -3)--(7, -3);
		\draw[blue](7, -3)--(8, -2);
		\draw[blue](8, -2)--(9, -1);
		\draw[blue](9, -1)--(10, 0);
	\end{tikzpicture}
	
	\caption{Polycycle with one turn.}
\end{figure}

with horizontal maps being in $\langle \mathcal{NC}\rangle$, maps down being $\exp$ and maps up being $\ln$, so by the shape of it adding the following lines standing for superfluous applications of $\exp$ and $\ln$ gives the correct form:

\begin{figure}[H]
	\begin{tikzpicture}
		\draw[help lines, color=gray!30, dashed] (0,-3.9) grid (9.9,1);
		\draw[->,ultra thick] (0,0)--(10,0) node[right]{$x$};
		\draw[->,ultra thick] (0,-4)--(0,1) node[above]{$y$};
		\draw[blue](0, 0)--(1, -1);
		\draw[red](1, -1)--(1, 0);
		\draw[blue](1, -1)--(2, -2);
		\draw[red](2, -2)--(2, 0);
		\draw[blue](2, -2)--(3,-2);
		\draw[red](3, -2)--(3, 0);
		\draw[blue](3, -2)--(4, -2);
		\draw[red](4, -2)--(4, 0);
		\draw[blue](4, -2)--(5, -2);
		\draw[red](5, -2)--(5, 0);
		\draw[blue](5, -2)--(6, -3);
		\draw[red](6, -3)--(6, 0);
		\draw[blue](6, -3)--(7, -3);
		\draw[red](7, -3)--(7, 0);
		\draw[blue](7, -3)--(8, -2);
		\draw[red](8, -2)--(8, 0);
		\draw[blue](8, -2)--(9, -1);
		\draw[red](9, -1)--(9, 0);
		\draw[blue](9, -1)--(10, 0);
	\end{tikzpicture}
	
	\caption{Splitting up a polycycle with one turn.}
\end{figure}
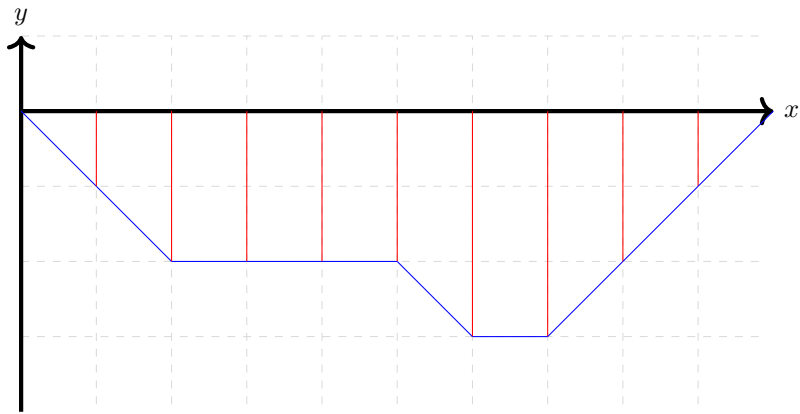

We stress that the $n$ in Theorem A is exactly the height of this polycycle as defined in \cite{Melvin24}.
\end{proof}

\subsection*{Structure of the work}
In Section \ref{Sec:cochains}, we define rigorously the cochains we work with and the terms used in the definition of the class $\mathcal{NC}$. Then we will go into the more technical parts of our work; that is, the repartitioning procedure we require; which strictly speaking is not necessary in the case we are working with, but it is necessary in general when we start working with the full Dulac's problem as in \cite{Melvin24}, so we have elected to nonetheless include it here.

Then we will extend the notion of Cauchy--Heine transform from summability; cf., \cite{BalserFormalODEs,summabilityLodayRichaud}, to a nice set of cochains, obviously including those relevant for Theorem A. Using that we will establish a version of the well-known Phragm\'en--Lindel\"{o}f principle; cf., \cite{PhragmenLindelof-article,theoryofFunctionsTitchmarsh}, which while it may look complicated formally, should essentially be seen as the following statement: Suppose that a cochain $f$ defined, in total, on some large domain $\Omega$ is small enough to satisfy Phragm\'en--Lindel\"{o}f for that domain (for an analytic function $f$ this would mean $f$ is identically zero), then $f$ is at most the size of their Stokes phenomena/coboundary.

Finally we will use these technical results and terminology to more precisely define the class $\mathcal{NC}$, and to prove the remaining Lemmas that allow us to prove Theorem A.

\subsection*{A historical note}
We would like to emphasize that our work has been largely influenced by the seminal work carried out first by Yu.S. Ilyashenko in his approach to Dulac's problem; cf., \cite{ilyashenkoFiniteness}; being the main results proved here generalizations of those that are present in \cite{ilyashenkoFiniteness}. Namely, the cochains we work with are generalized versions of the ones introduced in \cite[\S 1.1, 1.6]{ilyashenkoFiniteness}. Plus, the Cauchy--Heine transform in this article generalizes the one in \cite[\S 3.4B Lemma 1]{ilyashenkoFiniteness}, and the Theorem of Phragm\'en--Lindel\"{o}f we prove generalizes \cite[\S 3.6 Lemma 4]{ilyashenkoFiniteness}.

The proof of Theorem A is essentially as outlined in \cite[\S 3.2]{ilyashenkoFiniteness}, but in a much simpler case, as was the intention. In particular in \cite{ilyashenkoFiniteness} there is a nested (finite) sequence of partitions which had to be worked away using Phragm\'en--Lindel\"{o}f for cochains, while here we only have one. Moreover, we stress that the induction argument here does not work in general.

\subsection*{Acknowledgements}
We would like to thank Dmitry Novikov for the insightful comments he provided during the several discussions we had about our work. Particularly, we are grateful with him for explaining us Ilyashenko's version of Phragm\'en--Lindel\"{o}f principle.

\section{Cochains}\label{Sec:cochains}

Before introducing the cochains of \cite{ilyashenkoFiniteness}, we feel it appropriate to compare and contrast to its `close cousin' \v{C}ech cohomology. The biggest difference which informs the small practical differences is philosophical. For an overview of \v{C}ech cohomology we refer the reader to the classic text by R. Godement; cf., \cite[\S 5]{Godement-book}.

Generally speaking, given a topological space $X$, and an open cover of it $\mathcal{U}$, one of the main ideas behind \v{C}ech cohomology is to deduce global properties of $X$ by knowing local data in terms of $\mathcal{U}$, and how the open sets in $\mathcal{U}$ are glued together. It is in some sense a topological and combinatorial question. The biggest indicator being that this can without difficulty be generalized to sheaves of abelian groups.

When looking at cochains as in \cite{ilyashenkoFiniteness}, we actually start with a single analytic function, say real analytic. Classical Phragm\'en--Lindel\"{o}f actually asserts that the domain of analytic continuation of a real analytic function can preclude some degenerate behaviour; e.g., any oscillation for an real analytic function on the right half-plane still has to have 'peaks' that are of the size $e^{-\lambda x}$.

The question asked in \cite{ilyashenkoFiniteness} is essentially `what if your real analytic function has Stokes phenomenon beyond the domain of analytic continuation? Does this prevent degenerate behaviour?' In this line of inquiry a cochain is supposed to represent a function together with its Stokes phenomenon, unlike in \v{C}ech cohomology where a cochain is of independent interest.

The answer to this question is yes, which is what we will present as Phragm\'en--Lindel\"{o}f for cochains, the idea of the statement is the following: Suppose some real analytic function $f$ admits Stokes phenomenon up to some large domain $\Omega$, suppose on the real axis $f$ is small enough to apply Phragm\'en--Lindel\"{o}f on the large domain $\Omega$, then the size of $f$ on the real axis depends on two things:

\begin{enumerate}
    \item{The size and shape of the overlaps of the elements in covering on which $f$ is defined.}
    \item{The size of the differences on Stokes lines.}
\end{enumerate}
Both of these are things which are of no interest in \v{C}ech cohomology, so we will have to slightly redefine the notions from \v{C}ech cohomology to fit our purposes. Let us start with the idea of a partition:

\begin{Definition}
Let $\Omega \subset \mathbb{C}$ be a domain. A collection $\Xi$ of open subsets of $\Omega$ forms a \emph{partition} of $\Omega$ if:

\begin{enumerate}
    \item{The pairwise intersection of elements in $\Xi$ is empty.}
    \item{In the subspace topology on $\Omega$:
    \begin{equation*}
        \bigcup_{U \in \Xi} \overline{U} = \Omega\,.
    \end{equation*}
    }
    \item{This is locally finite in the sense that any point in $\Omega$ has an open neighbourhood containing only a finite amount of elements of $\Xi$.}
    \item{The boundary of each element of $\Xi$ is piecewise analytic.}
\end{enumerate}

We define $\partial \Xi$ to be the union of $\partial U$, $U \in \Xi$ (the boundary taken in $\Omega$).

Let $x \in \partial \Xi$, then we call $x$ a \emph{regular point} if there exists some open $A$ in $\Omega$, $x \in A$, such that $\partial \Xi \cap A$ is homeomorphic to a line. Else we call $x$ singular.

We call $\Xi$ a \emph{regular partition} if:

\begin{enumerate}
    \item{The singular points are isolated; i.e., for each singular point there is a neighbourhood containing no other singular points.}
    \item{Every singular point has at least one line of regular points going into it. We require that these lines have finite variation in argument; i.e., let $\gamma\colon [0, 1] \to \mathbb{C}$ be a parametrization of one of these curves of regular points going to the singular point $s$ at $t = 1$, then we want that $\arg(\gamma(t) - s)$ (with a branch cut of the logarithm in the origin) remains bounded as $t$ goes to $1$ (note that for this notion the choice of $\arg(\gamma(0) - s)$ does not matter).}
    \item{There exists a positive real number $d$ and a positive integer $n$ such that any ball of radius $d$ around any point has at most $n$ finite lines of regular points in it.}
\end{enumerate}

We call $d$ the \emph{regularity radius} and we call $n$ the \emph{multiplicity}.

We call $\Xi$ \emph{uniformly regular} if in addition there exists a $\delta > 0$ such that there is a distance $\geq \delta$ between a given singular point and all other singular points. We call the maximal $\delta$ the \emph{uniformity constant} of $\Xi$.
\end{Definition}

In order to catch the notion of `size of the overlaps' in a flexible way, that allows for things like Cauchy estimates, we introduce the notion of generalized $\epsilon$-neighbourhoods:

\begin{Definition}
Let $\Omega \subset \mathbb{C}$ be a domain, let $\Xi$ be a partition of $\Omega$, then a \emph{generalized $\epsilon$-neighbourhood} of $\Xi$ is given by:

\begin{enumerate}

    \item{A positive real number $\epsilon_{0}$.}
    \item{For each $\epsilon$ in $(0, \epsilon_{0})$ and each $U \in \Xi$ an open $U_{\epsilon}$ such that:
        \begin{enumerate}
            \item{For each $U \in \Xi$ and $\epsilon \in (0, \epsilon_{0})$ we have in the subspace topology of $\Omega$:
            \begin{equation*}
                \overline{U} \subset U_{\epsilon}\,.
            \end{equation*}
            }
            \item{If $\epsilon < \epsilon'$ and each $U \in \Xi$ we have:
            \begin{equation*}
                U_{\epsilon} \subset U_{\epsilon'}\,.
            \end{equation*}
            }
        \end{enumerate}}
\end{enumerate}
We will often shorten this to `Let $\Xi_{\epsilon}$ be a generalized $\epsilon$-neighbourhood of $\Xi$'.

Let $U \in \Xi$, then for $\epsilon > \epsilon' > 0$ we define the $\epsilon'-\epsilon$-\emph{diameter} of $U$ to be:
\begin{equation*}
    d_{U}(\epsilon', \epsilon) \coloneqq \sup\{r \geq 0 \mid \forall u \in U_{\epsilon'}, B(u, 2r) \subset U_{\epsilon}\},
\end{equation*}
where $B(u, 2r)$ is the ball around $u$ of radius $2r$. We define the $\epsilon'-\epsilon$-\emph{diameter of the partition} $\Xi$ with its generalized $\epsilon$-neighbourhood to be:
\begin{equation*}
    d_{\Xi}(\epsilon', \epsilon) \coloneqq \inf_{U \in \Xi}d_{U}(\epsilon', \epsilon)\,.
\end{equation*}

We say a generalized $\epsilon$-neighbourhood is \emph{regular} if $d_{\Xi}(\epsilon', \epsilon) > 0$ as long as $\epsilon$ is small enough.

Instead of $d_{\Xi}(0, \epsilon)$ or $d_{U}(0, \epsilon)$ we will just write $d_{\Xi}(\epsilon)$ and $d_{U}(\epsilon)$ respectively.
\end{Definition}

Let us talk about some obvious operations to perform with partitions:

\begin{Definition}
Let $\Xi$ and $\Xi'$ be partitions of $\Omega$, then their product is defined to be:
\begin{equation*}
    \Xi\cdot\Xi' \coloneqq (U \cap V)_{U \in \Xi, V \in \Xi'}
\end{equation*}
if this is a partition. We define the product of $\Xi_{\epsilon}$ and $\Xi'_{\epsilon}$, denoted $\Xi_{\epsilon}\cdot\Xi'_{\epsilon}$ by:
\begin{equation*}
    (U\cap V)_{\epsilon} \coloneqq U_{\epsilon} \cap V_{\epsilon}\,.
\end{equation*}
This is then a generalized $\epsilon$-neighbourhood for $\Xi \cdot \Xi'$.

Let $\rho\colon \Omega' \to \Omega$ be a biholomorphism. We define:
\begin{equation*}
    \rho^{-1}\Xi \coloneqq \{\rho^{-1}(U) \mid U \in \Xi\}\,.
\end{equation*}
We define the pullback of $\Xi_{\epsilon}$ by $\rho$, denoted $\rho^{-1}\Xi_{\epsilon}$ by:
\begin{equation*}
    (\rho^{-1}(U))_{\epsilon} \coloneqq \rho^{-1}(U_{\epsilon})\,.
\end{equation*}
This is then a generalized $\epsilon$-neighbourhood of $\rho^{-1}\Xi$.
\end{Definition}

With this notion of partition more suited to the study of Stokes phenomena we can look at cochains:

\begin{Definition}
Let $\Xi_{\epsilon}$ be a generalized $\epsilon$-neighbourhood on $\Omega$, a \emph{cochain} $f$ for this generalized $\epsilon$-neighbourhood consists of a positive real number $\epsilon$, and for each open set $U \in \Xi$ an analytic function $f_{U}$ on $U_{\epsilon}$.

We call $\Omega$ the \emph{total domain} of the cochain $f$. We also say that $f$ is an \emph{($\epsilon$-)extendable cochain}.

We may sometimes say `let $f$ be a cochain' in that case, the total domain is $\Omega^{f}$ and the corresponding partition is $\Xi^{f}$.
\end{Definition}

Let us talk about a few obvious operations we can do with cochains:

\begin{Definition}
Let $f$ and $g$ be cochains on $\Omega$. Then:
\begin{enumerate}
    \item{Their sum $f + g$ is a cochain with partition $\Xi^{f} \cdot \Xi^{g}$ and for $U \in \Xi^{f}$ and $V \in \Xi^{g}$ we have:
    \begin{equation*}
        (f + g)_{U \cap V} = f_{U}|_{U \cap V} + g_{V}|_{U \cap V}\,.
    \end{equation*}
    }
    \item{Their product $f\cdot g$ is a cochain with partition $\Xi^{f} \cdot \Xi^{g}$ and for $U \in \Xi^{f}$ and $V \in \Xi^{g}$ we have:
    \begin{equation*}
        (f\cdot g)_{U \cap V} = f_{U}|_{U \cap V}\cdot g_{V}|_{U \cap V}\,.
    \end{equation*}
    }
    \item{The derivative of $f$, denoted $f'$ is a cochain with partition $\Xi^{f}$ and:
    \begin{equation*}
        (f')_{U} = (f_{U})'\,.
    \end{equation*}
    }
\end{enumerate}

Let $\rho\colon \Omega' \to \Omega$ be a biholomorphism, then $f \circ \rho$ is defined on $\rho^{-1}\Xi^{f}$ by:
\begin{equation*}
    (f \circ \rho)_{\rho^{-1}(U)} = f_{U} \circ \rho\,.
\end{equation*}
\end{Definition}

\section{Repartitioning}

The reason for this part is that, those elements arising from the Additive Decomposition Theorem, as presented in \cite[cf., Theorem 2.14]{Melvin24}, will naturally occur on a regular partition; however, to get really much information we need uniformly regular partitions. The point of this section is to show that cochains on regular partitions with regular $\epsilon$-neighborhoods are also defined on a uniformly regular partition with regular $\epsilon$-neighbourhoods, preserving certain features we will be interested in. To get there, we will first need to introduce a way of rechoosing $\epsilon$-neighbourhoods.

\begin{Definition}
Suppose given a partition $\Xi$, then we can define the natural $\epsilon$-neighbourhoods associated to $\Xi$ taking for each $\epsilon > 0$ and each $U \in \Xi$:
\begin{equation*}
    U_{\epsilon} = B_{\epsilon}(U) \coloneqq \left\{z \in \mathbb{C} \mid \inf_{u \in U}|z - u| < \epsilon\right\}.
\end{equation*}
we call any cochain $f$ which is $\epsilon$-extendable in these neighbourhoods \emph{naturally $\epsilon$-extendable}.
\end{Definition}

\begin{Remark}
It is clear that if a cochain $f$ is extendable on a regular set of neighbourhoods, then it is naturally extendable.

In some sense, this emphasizes that our main concern is extendability in the natural sense. However, in practice, problems often arise with more apparent extensions—such as sectors to larger sectors or strips to larger strips; etc. Attempting to convert every such generalized neighborhood to the natural one is not only tedious but may also result in losing crucial information. For example, changing the opening of a sector affects the Phragmén–Lindelöf theorem for the domain, while adding balls of a given radius does not.
\end{Remark}

What we then want to prove is the following:

\begin{Theorem}\label{ThmRepart}
Let $f$ be a cochain naturally $\epsilon$-extendable on some regular partition $\Xi$ on some domain $\Omega$ contained in the right half-plane $\mathbb{C}^{+}$. Let $\epsilon$ be a positive real number smaller than the regularity radius. Then there exists a uniformly regular partition $\Xi'$ on the same total domain with uniformity constant $\frac{\epsilon}{3}$ such that the cochain $f$ is naturally $\frac{\epsilon}{3}$-extendable on $\Xi'$.

Moreover, the partition $\Xi'$ can be chosen such that outside of a distance $\frac{\epsilon}{6}$ from the singular points of $\Xi$, $\partial \Xi = \partial \Xi'$. Plus, the partition $\Xi'$ has the same regularity radius.
\end{Theorem}
	
\begin{proof}
The proof is essentially by local surgery as on the picture below:

\begin{figure}[htb]
    \includegraphics[scale = 0.7]{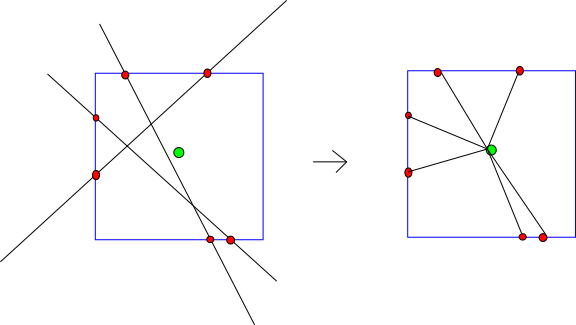}
    \centering
    \caption{Local surgery.}
\end{figure}

The red points remaining the same across both drawings and the green point being the centre.

The argument is roughly that we can keep performing the surgery above in a tiling of squares with radius $\frac{\epsilon}{6}$ (by radius we mean the distance from the centre to one of the sides) around singular points, resulting in a distance of $\frac{\epsilon}{3}$ between two singular points. By the singular points being isolated, this gives a new partition.

We may assume without loss of generality that there are no singular points on the boundary of any of the squares. Indeed, if we only consider tilings with a corner at the origin, for any fixed point there are at most a countable amount of tilings having this point on its boundary. So, because we only have a countable amount of singular points we can always choose a positive real number $\epsilon$ among the uncountable choices such that there are no singular points on the boundary of the tiling by squares with radius $\frac{\epsilon}{6}$.

We can have a look at the amount of points on the boundary of such a square, it is at most twice the multiplicity because any line segment can at most intersect the boundary twice, implying that a finite amount of lines go into each new singular point. Then if we take squares with the same radius as the old regularity radius over $\sqrt{2}$, it can only have a finite amount $m$ of such squares in it because any such square has a set radius and any of the $n$ old curves at worst breaks up into two curves with each ball, thus turning into $m + 1$ curves giving $n(m + 1)$ as multiplicity.

Hence, this partition $\Xi'$ clearly satisfies all properties ascribed to it. And because the size of the set we manipulate is at worst $\frac{\sqrt{2}\epsilon}{3}$ away from a boundary point; i.e., at worst a boundary we care about is on one boundary point and we need to drag it to the opposite boundary point, we still have natural $\frac{\epsilon}{3}$-extendability of our function $f$.
\end{proof}

\section{The Cauchy--Heine transform}

In this section, we introduce a concept of coboundary suitable for our purposes. To compare and contrast with \v{C}ech cohomology again, we do not care about the fact that going from cochain to coboundary gives a differential and later a complex. What we do care about is differences of our cochain on Stokes lines because we are interested in the qualitative difference between an analytic function on a domain and a cochain on that same domain. To that end, we also care about the coboundary as a single function on the total domain, because it, in a sense, quantifies the difference between being a function and a cochain.

\begin{Definition}
Let $\Xi$ be a regular partition. Then the \emph{coboundary} of an extendable cochain $f$ on the partition $\Xi$ is the tuple of analytic function $\delta_{U, V} f$ indexed by $(U, V) \in \Xi^{2}$ defined on $\partial U \cap \partial V$ by:
\begin{equation*}
    \delta_{U, V}f(z) \coloneqq f_{U}(z) - f_{V}(z) ,
\end{equation*}
together with the choice of $(U, V)$ we get a natural orientation on $\partial U \cap \partial V$ by saying that $U$ is on the left and $V$ is on the right.
\end{Definition}

Now we get to the part that allows us to relate a cochain to an analytic function on its entire domain.

\begin{Theorem}[Trivialization of a cocycle]\label{Theorem triv cocycle}
Let $\Omega \subset \mathbb{C}$ be a domain, let $\Xi$ be a regular partition, let $U_{\epsilon}$ be a regular choice of generalized $\epsilon$-neighbourhoods. Let $f$ be an $\epsilon$-extendable cochain, and let us take a choice $\delta f$ of coboundary.

Let us define for any path $A$ consisting only of the regular points of $\partial \Xi$:
\begin{equation*}
    \mathfrak{C}_{A}(f)(\zeta) \coloneqq \frac{1}{2}\sum_{(U, V) \in \Xi^{2}}\frac{1}{2\pi i}\int_{A}\frac{\delta_{U, V}f(\tau)}{\tau - \zeta}d\tau .
\end{equation*}
Here we integrate over $A$ using the appropriate orientation for $\delta_{U, V} f$ defined above, we assume that if we approach a singularity we are able to integrate by approaching the singularity; that is, the integral up to distance $\epsilon > 0$ away from the singularity converges to something. We have for every pair $(U, V) \in \Xi^{2}$ that $\delta_{U, V} \mathfrak{C}_{A}(f) = \delta_{U, V}f$ on the interior of $A$ (as subset of $\partial \Xi$).

Moreover, let $B$ be an open such that $A$ is the set of regular points in $B \cap \partial \Xi$. Suppose that:

\begin{enumerate}
    \item{We have:
    \begin{equation*}
        \frac{1}{2}\sum_{(U, V) \in \Xi^{2}}\frac{1}{2\pi}\int_{\partial \Xi}|\delta_{U, V}f(\tau)|d\tau < \infty .
    \end{equation*}
    }
    \item{For each $U$ and $V$ in $\Xi$ we have some $\epsilon > 0$ such that $|\delta_{U, V} f|$ is bounded up to $U_{\epsilon} \cap V_{\epsilon}$.}
    \item{Close enough to each singular point in $\partial \Xi \cap B$ $f$ is bounded in the sense that there is some neighbourhood $X$ of the singular point such that for all $W \in \Xi$, $f_{W}$ is bounded on $X \cap W_{\epsilon}$.}
\end{enumerate}

Then $f - \mathfrak{C}_{A}$ is analytic in the entirety of $B$; i.e., including on the singular points of $\partial\Xi$.

When there is no subscript $A$, then we assume $A = \partial \Xi$ and we call $\mathfrak{C}(f)$ the Cauchy--Heine transform of $f$.
\end{Theorem}

In order to prove this theorem we need to understand the function:
\begin{equation*}
    g(z) \coloneqq \frac{1}{2\pi i}\int_{\gamma}\frac{\phi(\zeta)}{z - \zeta}d\zeta \,,
\end{equation*}
for suitable $\gamma$ and $\phi$. Let us start with a very simple case inspired by similar proofs in the Gevrey asymptotics case; cf., \cite[Theorem 1.4.2]{summabilityLodayRichaud}:

\begin{Lemma}\label{Lemma CH trans}
Let $U$ be a open set in $\mathbb{C}$. Let $\phi\colon U \to \mathbb{C}$ be an analytic function. Let $\gamma$ be a smooth finite positive length curve in $U$ without self-intersection. Fix an orientation for $\gamma$, thus also fixing locally around $\gamma$ a sense of left and right. Let $a, b \in U$ be distinct points not on $\gamma$, let $\alpha$ be a finite length smooth curve from $a$ to $b$ with no self-intersection and intersecting $\gamma$ exactly once, going from left to right (see illustration below). Suppose that on $U\setminus \gamma$ the following function is defined:
\begin{equation*}
    g(z) \coloneqq \frac{1}{2\pi i}\int_{\gamma}\frac{\phi(\zeta)}{z - \zeta}d\zeta\,.
\end{equation*}
Then it is possible to analytically continue $g$ along the paths:

\begin{enumerate}
    \item{$\alpha$ from $a$ to $b$, we call this value $g^{+}(b)$.}
    \item{The inverse of $\alpha$ from $b$ to $a$, we call this value $g^{-}(a)$.}
\end{enumerate}

These values satisfy the following relations:
\begin{equation*}
    g^{+}(b) - g(b) = \phi(b)\,, \quad g(a) - g^{-}(a) = \phi(a)\,.
\end{equation*}
Moreover, it is possible to analytically continue $g$ to $\gamma$ from both sides of $\gamma$.
\end{Lemma}
\begin{figure}[htb]
    \includegraphics[scale = 0.35]{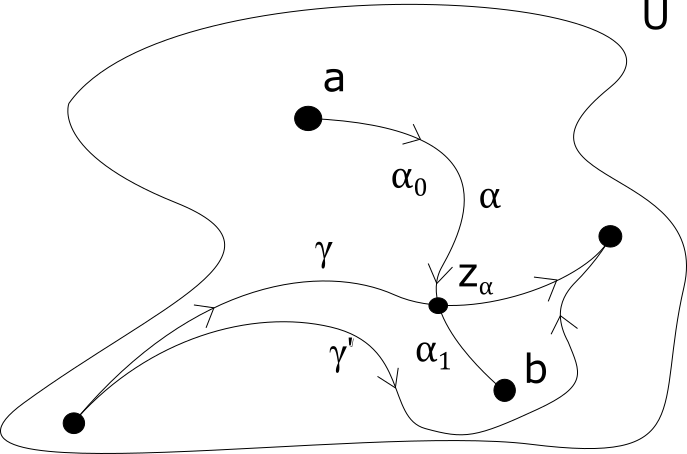}
    \centering
    \caption{Cauchy--Heine transform.}
\end{figure}

\begin{proof}
Let us start by noting that $g$ is analytic by combining the classic Theorems of Morera (\cite[p. 122]{ahlfors1966complex}) and Fubini and noting that the integrand for fixed $\zeta$ is analytic. In addition, let us note that by saying that $\gamma$ and $\alpha$ are smooth we mean that it has a smooth unit speed parametrization from some interval.

Suppose $\alpha$ intersects $\gamma$ in $z_{\alpha}$, this splits $\alpha$ into two parts, $\alpha_{0}$, the part before $z_{\alpha}$ and $\alpha_{1}$, the part after $z_{\alpha}$. If we want to analytically continue from $a$ to $b$ along $\alpha$, then we homotopically deform $\gamma$ `around $\alpha_{1}$, always keeping $\alpha_{1}$ to its left'. This is possible because $\gamma$ is smooth and $\alpha$ intersects $\gamma$ only once. Hence, we conclude that at any point of $\alpha_{1}$ there is an open not containing any points of $\gamma$, where we can deform $\gamma$ to wrap around $\alpha_{1}$ as described.

Call this new curve $\gamma'$, suppose $\gamma'$ deviates from $\gamma$ at $z_{0}$ close to $z_{\alpha}$ and rejoins $\gamma$ at $z_{1}$, on the other side of $z_{\alpha}$.

Note that, by construction, for any point $z$ on $\alpha_{0}$ we have:
\begin{equation*}
    \frac{1}{2\pi i}\int_{\gamma}\frac{\phi(\zeta)}{z - \zeta}d\zeta = \frac{1}{2\pi i}\int_{\gamma'}\frac{\phi(\zeta)}{z - \zeta}d\zeta ,
\end{equation*}
and because $\alpha_{1}$ touches $\gamma'$ nowhere by construction, the following function:
\begin{equation*}
    g^{+}(z) \coloneqq \frac{1}{2\pi i}\int_{\gamma'}\frac{\phi(\zeta)}{z - \zeta}d\zeta
\end{equation*}
remains analytic along $\alpha$ for the same reason as $g$, so $g^{+}(b)$ satisfies the demands.

Then:
\begin{equation*}
    g^{+}(b) - g(b) = \frac{1}{2\pi i}\int_{\gamma'}\frac{\phi(\zeta)}{b - \zeta}d\zeta - \frac{1}{2\pi i}\int_{\gamma}\frac{\phi(\zeta)}{b - \zeta}d\zeta .
\end{equation*}
This is just the contour integral of $\frac{\phi(\zeta)}{b - \zeta}$ from $z_{0}$ to $z_{1}$ following $\gamma'$ and then back from $z_{1}$ to $z_{0}$ along $\gamma$. By construction this curve goes counter-clockwise, contains $b$ in its interior, and by assumption $\phi$ has no singularities in $U$, so by Cauchy's Theorem we have:
\begin{equation*}
    g^{+}(b) - g(b) = \phi(b) .
\end{equation*}
The other cases are analogous. Note that in order to extend $g$ to $\gamma$, it suffices to pick a point on $\gamma$ and because $\gamma$ has no self-intersection and is smooth, it is possible to make a small curve perpendicular to $\gamma$, intersecting $\gamma$ only at the chosen point and one uses this small curve to analytically continue.
\end{proof}

Before proving Theorem \ref{Theorem triv cocycle}, we still need a second auxiliary result to treat the singular points. That said, we remark that this Lemma will be used a lot independently later.

\begin{Lemma}\label{Lemma Unif CH estimates}
Let $f$ be an extendable cochain for a regular partition $\Xi$ and a regular set of generalized $\epsilon$-neighbourhoods, let $f$ be $\epsilon$-extendable and let $\zeta_{0}$ be a point in $U$ in the partition of $f$. Suppose that:
\begin{enumerate}
    \item{the integral:
    \begin{equation*}
        \frac{1}{2}\sum_{(U, V) \in \Xi^{2}}\frac{1}{2\pi}\int_{\partial \Xi}|\delta_{U, V}f(\tau)|d\tau < \infty\,.
    \end{equation*}
    }
    \item{We have some $d > 0$, $d \leq d_{\Xi}(\epsilon)$ such that $|\delta f|$ is bounded on $\overline{B(\zeta_{0}, d)}$ in the sense that it is bounded by the same constant over all choices of two elements of the partition.}
    \item{Let $L < \infty$ be the sum of for each regular piece of boundary both entering and exiting $\overline{B(\zeta_{0}, d)}$ the difference between entering and leaving angle, in radians.}
    \item{Suppose that any regular piece of boundary entering $\overline{B(\zeta_{0}, d)}$ also leaves; i.e., $\overline{B(\zeta_{0}, d)}$ contains no singular points.}
\end{enumerate}
Then the following inequality takes place:
\begin{equation*}
    \left|\mathfrak{C}(f)(\zeta_{0})\right| \leq \frac{\frac{1}{2}\sum_{(U, V) \in \Xi^{2}}\frac{1}{2\pi}\int_{\partial \Xi}|\delta_{U, V}f(\tau)|d\tau+ L\sup_{\zeta_{1} \in \overline{B(\zeta_{0}, d)}}|\delta f(\zeta_{1})|}{2\pi d}\,.
\end{equation*}
In particular, should you have an $W \subset \Omega^{f}$ such that:
\begin{enumerate}
    \item{$|\delta f|$ is bounded on $\overline{B(W, d)}$ (the set of all points at distance $\leq d$ from $W$).}
    \item{We have a uniform bound $L$ of the lengths $L$ for all points. And $\overline{B(W, d)}$ still contains no singular points.}
\end{enumerate}
Then:
\begin{equation*}
    \sup_{\zeta_{0} \in W}\left|\mathfrak{C}(f)(\zeta_{0})\right| \leq \frac{\frac{1}{2}\sum_{(U, V) \in \Xi^{2}}\frac{1}{2\pi}\int_{\partial \Xi}|\delta_{U, V}f(\tau)|d\tau + L\sup_{\zeta_{1} \in \overline{B(W, d)}}|\delta f(\zeta_{1})|}{2\pi d}\,.
\end{equation*}
That is, the previous estimate becomes uniform, including on the boundary and can `cross the boundary'.
\end{Lemma}
\begin{proof}
Note that the second statement follows from the first by continuity.

For the first statement, let us recall the definition of a Cauchy--Heine transform:
\begin{equation*}
    \mathfrak{C}_{A}(f)(\zeta_{0}) \coloneqq \frac{1}{2}\sum_{(U, V) \in \Xi^{2}}\frac{1}{2\pi i}\int_{A}\frac{\delta_{U, V}f(\tau)}{\tau - \zeta}d\tau\,.
\end{equation*}
We get our estimate by splitting the integral into two parts, the part where $|\tau - \zeta_{0}| > d$ and the part where $|\tau - \zeta_{0}| \leq d$. Let $V_{1}$ denote the first part and $V_{2}$ the second, then:
\begin{align*}
    \left|\frac{1}{2}\sum_{(U, V) \in \Xi^{2}}\frac{1}{2\pi i}\int_{V_{1}}\frac{\delta_{U, V}f(\tau)}{\tau - \zeta}d\tau\right|
    &\leq \frac{1}{2}\sum_{(U, V) \in \Xi^{2}} \frac{1}{2\pi }\int_{V_{1}}\frac{|\delta_{U, V} f(\tau)|}{d}d\tau \\
    &\leq \frac{1}{2}\sum_{(U, V) \in \Xi^{2}} \frac{1}{2\pi }\frac{\int_{\partial \Xi^{f}}|\delta_{U, V} f(\zeta)|d\zeta}{2\pi d}\,.
\end{align*}

As for $V_{2}$, it is easy to estimate the parts both entering and leaving each region by deforming the path to the boundary resulting in:
\begin{equation*}
    \frac{L\sup_{\zeta_{1} \in \overline{B(\zeta_{0}, d)}}|\delta f(\zeta_{1})|}{2\pi d}\,.
\end{equation*}
\end{proof}

With these results in our hands, we can now prove Theorem \ref{Theorem triv cocycle}:

\begin{proof}[Proof of Theorem \ref{Theorem triv cocycle}]
All we really still need to prove is about the singular points. Let $s$ be a singular point in $\overline{A}$ inside $B$, then note that $s$ is an isolated singularity of $f - \mathfrak{C}_{A}(f)$, we can then use the estimate of Lemma \ref{Lemma Unif CH estimates} on small circles of radius $2d$ around $s$, denoted $S_{2d}$. Then because there are only a finite amount of lines of regular points going out of $s$ eventually we get an estimate for $|f - \mathfrak{C}_{A}(f)|$ of the following form:
\begin{equation*}
    \sup |f| + \left|\frac{ \frac{1}{2}\sum_{(U, V) \in \Xi^{2}}\frac{1}{2\pi}\int_{\partial \Xi}|\delta_{U, V}f(\tau)|d\tau + L\sup_{\zeta_{1} \in \overline{B(S_{2d}, d)}}|\delta f(\zeta_{1})|}{2\pi d}\right| .
\end{equation*}
But this implies that
\begin{equation*}
    \lim_{z \to s}(z - s)^{2}(f - \mathfrak{C}_{A}(f)) = 0\,.
\end{equation*}
So, by Riemann's Theorem; cf., \cite[4.3.1 Theorem 7]{ahlfors1966complex}, we get that $f - \mathfrak{C}_{A}(f)$ has at worst a simple pole at $s$. That is to say, that $\lim_{z \to s}(z - s)(f - \mathfrak{C}_{A}(f))$ exists. Suppose that it does have a simple pole at $s$, then this limit should be non-zero. But a simple calculation through Taylor series shows that the integral:
\begin{equation*}
    \int\frac{\delta(f)(w)}{w - z}dw
\end{equation*}
behaves roughly as a logarithm times $\delta(f)(z)$ when $w \to z$, so we are roughly looking at $z\ln(z)$ around the origin, which goes to zero unless $z$ spirals around the origin fast enough. But, by regularity, we assumed finite variation of argument for the lines along which we integrate, so we can just approach the singular point radially and have $\lim_{z \to s}(z - s)(f - \mathfrak{C}_{A}(f)) = 0$. Thus, once again Riemann's Theorem gets us where we want to be.
\end{proof}

From this and the classical Phragm\'en--Lindel\"{o}f (cf., Theorem \ref{Thm:classical P-L} below), one quickly gets a form of maximum modulus principle for cochains:

\begin{Corollary}\label{Cor triv of cocycle}
Let $f$ be a bounded cochain satisfying all conditions of Theorem \ref{Theorem triv cocycle} with $A$ being all the regular points of $\Xi$. Suppose $\Omega$ is biholomorphic with $\mathbb{C}^{+}$ through $\rho\colon \mathbb{C}^{+}\to \Omega$, mapping the boundary of $\Omega$ to the imaginary axis. Suppose $f$ and $\mathfrak{C}(f)$ are bounded on $\Omega$. Then:
\begin{equation*}
    \sup_{\Omega}|f| \leq \sup_{\partial \Omega}|f - \mathfrak{C}(f)| + \sup_{\Omega}|\mathfrak{C}(f)| \leq \sup_{\partial \Omega}|f| + 2\sup_{\Omega}|\mathfrak{C}(f)|\,.
\end{equation*}
\end{Corollary}

\begin{Remark}
This Corollary obviously also works when one cuts out a compact part of $\mathbb{C}^{+}$.
\end{Remark}

Combining this with Lemma \ref{Lemma Unif CH estimates} when we have a uniformly regular partition we get the estimate:

\begin{Corollary}\label{CorFinEstCH}
Let $f$ be a bounded cochain satisfying all conditions of Theorem \ref{Theorem triv cocycle} with $A$ being all the regular points of $\Xi$ and $\Xi$ being uniformly regular. Suppose $\Omega$ is biholomorphic with $\mathbb{C}^{+}$ through $\rho\colon \mathbb{C}^{+}\to \Omega$, mapping the boundary of $\Omega$ to the imaginary axis. Then there exists a constant $C$ depending only on the partition (that is, how well we can control the estimate of Lemma \ref{Lemma Unif CH estimates} for singular points using Jordan curves around them) such that:
\begin{equation*}
    \sup_{\Omega}|f| \leq \sup_{\partial \Omega}|f| + C\left(\int_{\partial \Xi}|\delta f(\zeta)|d\zeta + \sup_{\zeta_{1} \in \Omega}|\delta f(\zeta_{1})|\right).
\end{equation*}
\end{Corollary}
\begin{proof}
As soon as we know $f - \mathfrak{C}(f)$ is analytic we know that every component of $\mathfrak{C}(f)$ is analytic even extended past a singular point, so we can use maximum modulus Theorem on a small circle around the singular point in the boundary inside the $\epsilon$-neighbourhood of a chosen component, by uniform regularity it is possible to take the size of this circle the same for all singular points, resulting in a uniform estimate using Lemma \ref{Lemma Unif CH estimates}.

Perhaps more clearly we pick an open set $U$ bordering our singular point $s$, then $U_{\epsilon}$ contains $s$ in its interior and we can analytically continue $\mathfrak{C}(f)_{U}$ around $s$ in a small circle $S$, because $f - \mathfrak{C}(f)$ is fully analytic and $f_{U}$ is analytic up to $U_{\epsilon}$ we know that analytically continuing $\mathfrak{C}(f)_{U}$ around $s$ does not create a branch cut, but rather comes back to our original function. Then we apply maximum modulus Theorem to estimate the modulus of $\mathfrak{C}(f)_{U}(s)$ by its value on $S$, which can be bounded in the same way as Lemma \ref{Lemma Unif CH estimates}, by exactly the same proof. By regularity of the partition, we can take the same $S$ for all $U$ bordering on $s$; and, by uniform regularity, we can take circles of the same radius for all singular points, fixing all constants in Lemma \ref{Lemma Unif CH estimates}.
\end{proof}

\section{Phragm\'en--Lindel\"{o}f principle}
Before we move on to what we will refer to as Phragm\'en--Lindel\"{o}f for regular cochains, we will need an improvement to classical Phragm\'en--Lindel\"{o}f principle; cf., \cite{PhragmenLindelof-article}. With this aim, next we recall the classical Phragm\'en--Lindel\"{o}f principle following Titchmarsh text \cite{theoryofFunctionsTitchmarsh}:

\begin{Theorem}[Classical Phragm\'en--Lindel\"{o}f, {\cite[\S 5.61]{theoryofFunctionsTitchmarsh}}]\label{Thm:classical P-L}
Let $f$ be an analytic function in an unbounded simply connected region $U$ inside a sector at infinity making angle $\frac{\pi}{\alpha}$ at infinity, including the boundary; i.e., $f$ is analytic on $\overline{U}$, and $\overline{U}$ is contained inside the closure of the sector.

Suppose that on $\partial U$:
\begin{equation}\label{Eq:PhragmenLindelof-inequality}
    |f(z)| \leq M .
\end{equation}
If there exists some $\beta < \alpha$ such that on $U$:
\begin{equation*}
    |f(z)| = O\left(e^{|z|^{\beta}}\right),
\end{equation*}
then Inequality \eqref{Eq:PhragmenLindelof-inequality} holds on the entirety of $U$.
\end{Theorem}

\begin{Remark}
The region $U$ was not in the original statement in \cite{theoryofFunctionsTitchmarsh}, but is clear that exactly the same proof works.
\end{Remark}

\begin{Definition}
Let $K$ be a compact subset of the right half-plane $\mathbb{C}^{+}$, and let $\gamma$ be a curve contained in $\mathbb{C}^{+}\setminus K$. We call $\gamma$ a \emph{meaningful dividing line} if $\mathbb{C}^{+}\setminus (K \cup \gamma)$ consists of two path connected components, each containing some sector with a positive angle.
\end{Definition}

A direct consequence of the classical Phragm\'en--Lindel\"{o}f, as stated in Theorem \ref{Thm:classical P-L}, is the following:

\begin{Proposition}
Let $K$ be a compact subset of the right half-plane $\mathbb{C}^{+}$, and let $\gamma$ be a meaningful dividing line in $\mathbb{C}^{+}\setminus K$. Let $f$ be an analytic function on $\overline{\mathbb{C}^{+}\setminus K}$, such that:
\begin{equation}\label{Eq:Refined-PL-Inequality}
    |f(z)| \leq M ,\quad\text{for any point}\quad z\in \gamma \cup \partial(\mathbb{C}^{+}\setminus K) .
\end{equation}
Suppose that there exists some $\lambda > 0$ such that on $\mathbb{C}^{+}\setminus K$:
\begin{equation*}
    |f(z)| \leq e^{\lambda |z|} ,
\end{equation*}
then Inequality \eqref{Eq:Refined-PL-Inequality} holds on $\mathbb{C}^{+}\setminus K$.
\end{Proposition}
\begin{proof}
It is enough to apply Theorem \ref{Thm:classical P-L} on each of the path connected components of $\mathbb{C}^{+}\setminus (K \cup \gamma)$.
\end{proof}

Once we have recalled the classical Phragm\'en--Lindel\"{o}f, we turn to the main result of this section; namely, its generalization for regular cochains.

\begin{Theorem}[Phragm\'en--Lindel\"{o}f for regular cochains]\label{ThmPLCoch}
Let $f$ be a bounded cochain satisfying all conditions of Corollary \ref{CorFinEstCH}. Let $\phi\colon \Omega \to \mathbb{C}^{+}\setminus K$ be a biholomorphism, $K$ a compact set. Suppose that the image of the real axis under $\phi$ is a meaningful dividing line.

Define the set $B_{a}$ as all $z$ in the boundary of $\Xi$ with real part $a$. Let $J\colon B_{a} \to \mathbb{R}$ be the map mapping a regular point $z$ to the following: locally there exists (if it does not exist, make $J(z)=\infty$, implies a need to repartition usually) a function $\gamma$ such that $\partial \Xi$ is given by the points $\gamma(t)$ with $\gamma(0) = z$ and $\re(\gamma(t)) = \re(z) + t$ then we define:
\begin{equation*}
    J(z) = |\gamma'(0)| .
\end{equation*}
If it is a singular point, $J(z) = 0$.

Let $L$ be an increasing positive $C^{1}$ function such that:
\begin{equation*}
    \sum_{z \in B_{a}}J(a) \leq L(a) .
\end{equation*}

Suppose that $M$ is some positive non-zero $C^{1}$ function such that:
\begin{equation*}
    \sup_{\re(z) \geq a}|\delta f| \leq M(a) .
\end{equation*}

Suppose that $\rho$ is some positive non-zero $C^{1}$ function with positive derivative such that:
\begin{equation*}
    \sup_{\re(z) = a}\re(\phi(z)) \leq \rho(a) .
\end{equation*}

Suppose that $\sigma$ is some positive continuous function such that:
\begin{equation*}
    \sigma(a) \leq \inf_{\re(z) \geq a}\re(\phi^{-1}(z)) .
\end{equation*}

Suppose that:

\begin{enumerate}
    \item{The cochain $f$ is bounded on $\overline{\Omega}$.}
    \item{The cochain $f \circ \phi^{-1}$ descends in absolute value faster than any exponential on the image of the real axis under $\phi$.}
    \item{The function:
    \begin{equation*}
        \lambda(a) \leq \inf_{b \geq \sigma(a)} - \frac{\frac{M'(b)}{M(b)} + \frac{L'(b)}{L(b)}}{\rho'(b)}
    \end{equation*}
    is eventually positive.}
    \item{Assume that for all $a$:
    \begin{equation*}
        I_{\rho}(a) \coloneqq \int_{\sigma(a)}^{+\infty}e^{-\delta(a)(\rho(s) - \rho(\sigma(a)))}ds < \infty .
    \end{equation*}
    }
\end{enumerate}

Then for any real function $\psi$ going to $+\infty$ and smaller than $\re(\phi)$ on the real axis and for any positive function $\tilde{\delta}$ eventually between the zero function and $\lambda$, we have for $x \in \mathbb{R}$ large enough a constant $C$ depending only on the partition such that:
\begin{equation*}
    \hspace{-1cm}|f(x)| \leq e^{-\left[\lambda(\psi(x))- \tilde{\delta}(\psi(x))\right](\re(\phi(x)) - \psi(x))}\left(\sup_{ \Omega}|f| + CM(\sigma(\psi(x)))\left[L(\sigma(\psi(x)))I_{\rho}(\psi(x)) +  1\right]\right) .
\end{equation*}
\end{Theorem}
\begin{proof}
Let us define for $a > 0$:
\begin{equation*}
    \Omega_{a} \coloneqq \phi^{-1}(\{\re(z) > a\})\,,\quad f_{a} \coloneqq e^{(\lambda(a) - \tilde{\delta}(a))(\phi(z) - a)}f(z) .
\end{equation*}
We are now interested in estimating $f_{a}$ on $\Omega_{a}$ using Corollary \ref{CorFinEstCH}. As one might notice, the intention of this setup is that $\sup_{\partial \Omega_{a}}|f_{a}| = \sup_{\partial\Omega_{a}}|f| \leq \sup_{\Omega}|f|$.

Let us start by estimating the relevant integral for $a$ large enough to make $\lambda(a)- \tilde{\delta}(a)$ positive (to keep notation light we will omit the sum over all $(U, V) \in \Xi^{2}$ and the subscripts for $\delta$):
\begin{equation*}
    \int_{\partial \Xi \cap \Omega_{a}}|\delta f_{a}(\zeta)|d\zeta \leq \int_{\sigma(a)}^{+\infty}L(s)M(s)e^{\left[\lambda(a)- \tilde{\delta}(a)\right](\rho(s) - a)}ds .
\end{equation*}
Denote the integrand by $I(s)$, we want to apply Gronwall's Lemma as follows: note that by definition of $\lambda$:
\begin{equation*}
    \frac{I'(s)}{I(s)} = \frac{L'(s)}{L(s)} + \frac{M'(s)}{M(s)} + (\lambda(a) - \tilde{\delta}(a))\rho'(s) \leq -\tilde{\delta}(a)\rho'(s) .
\end{equation*}
This implies that:
\begin{equation*}
    L(s)M(s)e^{\left[\lambda(a)- \tilde{\delta}(a)\right](s - a)} \leq L(\sigma(a))M(\sigma(a))e^{-\tilde{\delta}(a)(\rho(s) - \rho(\sigma(a)))} .
\end{equation*}
Thus:
\begin{equation*}
    \int_{\partial \Xi \cap \Omega_{a}}|\delta f_{a}(\zeta)|d\zeta \leq L(\sigma(a))M(\sigma(a))I_{\rho}(a) .
\end{equation*}

Now we can look at:
\begin{equation*}
    \sup_{\Omega_{a}}|\delta f_{a}|\,,
\end{equation*}
and note that:
\begin{equation*}
    \sup_{\re(z) = s}|\delta f_{a}| \leq \frac{I(s)}{L(s)}\,.
\end{equation*}
Recall that we have assumed $L$ to be increasing, moreover, it is clear that $I(s) \geq 0$ thus already having calculated that $I'(s) \leq -\tilde{\delta}(a)\rho'(s)I(s)$ we know that $I(s)$ at least does not increase, thus:
\begin{equation*}
    \sup_{\Omega_{a}}|\delta f_{a}| \leq M(\sigma(a)).
\end{equation*}
Note that because $f$ is bounded, the improved Phragm\'en--Lindel\"{o}f in the Proposition above still works for $f_{a}\circ \phi^{-1}$. This combines with Corollary \ref{CorFinEstCH} into:
\begin{equation*}
    \sup_{\Omega_{a}}\left|e^{\left[\lambda(a)- \tilde{\delta}(a)\right](\phi(z) - a)}f(z)\right| \leq \sup_{ \Omega}|f| + C\left[M(\sigma(a))L(\sigma(a))I_{\rho}(a) + M(\sigma(a))\right].
\end{equation*}
Dividing by the exponential on the left hand side and removing the supremum we get:
\begin{equation*}
    \left|f(z)\right| \leq \left|e^{-\left[\lambda(a)- \tilde{\delta}(a)\right](\phi(z) - a)}\right|\left(\sup_{ \Omega}|f| + CM(\sigma(a))\left[L(\sigma(a))I_{\rho}(a) + 1\right]\right).
\end{equation*}
This estimate holds for $f \in \Omega_{a}$ now certainly for $x \in \mathbb{R}$, $x \in \Omega_{\re(\phi(x))} \subset \Omega_{\psi(x)}$ so filling in $x$ for $z$ and $\psi(x)$ for $a$ we get the estimate we want.
\end{proof}

\section{Proving Theorem A}

\begin{Definition}
Define on the right half-plane $\mathbb{C}^{+}$ the so-called \emph{simple standard partition} $\Xi_{st}$, given by the lines:
\begin{equation*}
    \im(\zeta) = \left\{n\frac{3}{4}\pi \mid n \neq 0\right\}.
\end{equation*}
This is given the generalized $\epsilon$-neighbourhoods by `enlarging the strips vertically by $\epsilon$ on both sides'.

We denote for the rest of this section the element of $\Xi_{st}$ containing the real axis by $U_{st}$.
\end{Definition}

\begin{Definition}
A \emph{simple cochain} $f$ (of type $1$) is given by a cochain on some:
\begin{equation*}
    \mathbb{C}^{+}_{a} \coloneqq \{\zeta \in \mathbb{C} \mid \re(\zeta) > a\},
\end{equation*}
with partition $\Xi_{st}$ such that for some $C, C' > 0$, for all $\zeta \in \mathbb{C}^{+}_{a}$:
\begin{equation*}
    |\delta f(\zeta)| \leq Ce^{-C'e^{\re(\zeta)}} .
\end{equation*}
\end{Definition}

\begin{Definition}\label{DefredefNC}
An element of the class $\mathcal{NC}$ is a real analytic function which can be extended to a simple cochain such that there exists some series:
\begin{equation*}
    \zeta + \sum a_{n}e^{-n\zeta} ,
\end{equation*}
with real coefficients $a_{n}$, such that any finite sum $S_{N}$ up to some $N$ approximates $f$ uniformly up to accuracy $O(e^{-(n + 1)\zeta})$; i.e., there exists some $C > 0$, $\epsilon > 0$, and some $\xi_{0} > 0$, such that for all $U \in \Xi_{st}$, for all $\zeta \in U_{\epsilon}$ with $\re(\zeta) > \xi_{0}$:
\begin{equation*}
    |f_{\Pi}(\zeta) - S_{N}(\zeta)| < Ce^{-(n + 1)\re(\zeta)} .
\end{equation*}
By abuse of notation we will use $f$ to both refer to the real germ as well as a chosen and fixed extension to a simple cochain.
\end{Definition}

Before we prove Theorem A, let us prove two auxiliary Lemmas:

\begin{Lemma}\label{LemPLSimpCochain}
Let $f$ be a simple cochain which is real on the real axis, suppose that $f$ is smaller than any exponential on the real axis, then $f$ is identically zero on the real axis.
\end{Lemma}
\begin{proof}
By applying Phragm\'en--Lindel\"{o}f for cochains on the smaller total domain, say given by:
\begin{equation*}
    \Omega \coloneqq \Psi(\mathbb{C}^{+}),
\end{equation*}
with $\Psi(\zeta) = \zeta + \sqrt{\zeta + 1}$ (one of the standard quadratic domains of \cite[p. 22]{ilyashenkoFiniteness}), it is easy to check that the boundary of $\Omega$ is given by:
\begin{equation*}
    it + \sqrt{1 + it} = \sqrt{\frac{\sqrt{1 + t^{2}} + 1}{2}} + i\left(\sgn(t)\sqrt{\frac{\sqrt{1 + t^{2}} - 1}{2}} + t\right).
\end{equation*}
We can take $L(a) = C_{1}a^{2}$, $M(a) = C_{2}e^{-C_{3}e^{a}}$ (in the notation of Theorem \ref{ThmPLCoch} with $C_{1}, C_{2}$ some constants). Also clearly $\Psi^{-1}$ plays the role of the function $\phi$ in the notation of the Phragm\'en--Lindel\"{o}f Theorem. So we can take $\sigma(a) = a$.

Plus, certainly $\re(z) \leq \re(\phi^{-1}(z))$ on the imaginary axis, now $|\re(z) - \re(\phi^{-1}(z))| = o(z)$, thus by classical Phragm\'en--Lindel\"{o}f for harmonic functions; cf., \cite[4. Corollary]{AhlforsPL} we can take:
\begin{equation*}
    \rho(a) = a\,.
\end{equation*}
Then:
\begin{equation*}
    \inf_{b \geq a}-\left(-C_{3}e^{b} + \frac{2}{b}\right) = \inf_{b \geq a}C_{3}e^{b} - \frac{2}{b} = C_{3}e^{a} - \frac{2}{a} ,
\end{equation*}
which we can simply take as $\lambda(a)$. Take $\delta(a) = 1$. Then:
\begin{equation*}
    I_{\rho}(a) = \int_{a}^{\infty}e^{-(x - a)}dx = 1\,.
\end{equation*}
We can also take $\psi(x) = \frac{1}{2}x$ because from the fact that $\Psi$ is near identity on the real axis the same can be deduced for $\Psi^{-1}$. From this we also get for $x$ large enough $\re(\Psi^{-1}(x)) - \psi(x) > \frac{1}{3}x$. This results in the following estimate on the real axis:
\begin{equation*}
    |f(x)| \leq e^{-\left(C_{3}e^{\frac{1}{2}x} - \frac{4}{x} - 1\right)\frac{1}{3}x}\left(C_{4} + C_{5}\left(C_{3}e^{\frac{1}{2}x} - \frac{4}{x} - 1\right)\left[\frac{x^{2}}{4} + 1\right]\right).
\end{equation*}
Simplifying down to the important parts and renaming the constants we get:
\begin{equation*}
    |f(x)| \leq Ce^{-C'xe^{\frac{1}{2}x}}\,.
\end{equation*}
Consequently, we note that the component containing the real axis is defined on the strip with imaginary parts between $-\pi$ and $\pi$. So, the composition $f \circ m_{2} \circ \ln $, with $m_{2}$ being multiplication by $2$, is defined on $\mathbb{C}^{+}$ outside some compact set, and we have that:
\begin{equation*}
    |f(x)| \leq Ce^{-C'\ln(x)x},
\end{equation*}
which is smaller than any exponential so by Classical Phragm\'en--Lindel\"{o}f this is identically zero.
\end{proof}

\begin{Lemma}\label{Lemma:NC group struc}
$\mathcal{NC}$ forms a group under composition.
\end{Lemma}
\begin{proof}
To see for the composition that both standard quadratic domains and Dulac series are preserved, see \cite[Lemma 24.33]{Ilyashenko08lectureson} (nothing is really changed by the cochain nature).

The argument that the partition can be preserved comes from its extendability to larger strips and the fact that the cochain is exponentially close to the identity, meaning that if $f_{1}, f_{2} \in \mathcal{NC}$, for $\re(\zeta)$ large enough, for any strip $\Pi$, $f_{1, \Pi}$ (the component of $\Pi$) will map $\Pi_{\epsilon/2}$ inside $\Pi_{\epsilon}$, which is the domain of $f_{2, \Pi}$, so we can just take the same partition by reducing the $\epsilon$ up to which we can extend to $\Pi_{\epsilon}$.

To show that $\mathcal{NC}$ is closed under inversion, we will use the following formula for the inverse function, derived from Rouch\'e's Theorem; cf., \cite[pp. 153--154]{ahlfors1966complex} (with thanks to Dmitry Novikov): Let $f$ be an invertible analytic function, note then that by Cauchy's Theorem:
\begin{equation*}
    f^{-1}(w_{0}) - w_{0} = \frac{1}{2\pi i}\int\frac{f^{-1}(w) - w}{w - w_{0}}dw .
\end{equation*}
Thus, making the substitution $w = f(z)$, we get:
\begin{equation*}
    f^{-1}(w_{0}) - w_{0} = \frac{1}{2\pi i}\int\frac{(z - f(z))f'(z)}{f(z) - w_{0}}dz .
\end{equation*}

Let us now take $\alpha \in \mathcal{NC}$ and consider $f = \id + \alpha$. We obtain:
\begin{equation*}
    (\id + \alpha)^{-1}(w_{0}) - w_{0} = \frac{1}{2\pi i}\int\frac{\alpha(z)(1 + \alpha'(z))}{z + \alpha(z) - w_{0}}dz,
\end{equation*}
noting that performing Cauchy estimates on circles of radius $e^{-\re(z)^{\frac{1}{2}}}$ will still preserve our domains and we get both the domain and the estimates we want. We only need to prove that $\id + \alpha$ is invertible for $\re(z)$ large enough, for the series expansion we refer to the formula in \cite[Lemma 6.23]{VdDriesCompTrans}.

Let $f(z) = z + \alpha(z)$, and let us then consider $\tilde{f}(z) = f(z + 2x_{0}) - 2x_{0}$, $\tilde{\alpha} = \alpha(z + 2x_{0})$. Note that by definition:
\begin{equation*}
    \tilde{f}(z) = z + \tilde{\alpha}(z) ,
\end{equation*}
but unlike with $\alpha$, $\tilde{\alpha}^{i}$; i.e., the $i$-fold composition, makes sense and we can explicitly give the inverse of $\tilde{f}$ as:
\begin{equation*}
    \tilde{f}^{-1}(w) = \sum_{i = 0}^{\infty}\tilde{\alpha}^{i}(w).
\end{equation*}
This implies that $z \mapsto \tilde{f}(z) - 2x_{0} = f(z + 2x_{0})$ will eventually be injective. Moreover, because $f'$ will remain non-zero by Cauchy estimates, we know that $f$ will eventually be injective on the type of domain that is necessary.
\end{proof}

Let us finally prove Theorem A; the main result of this manuscript. For this sake, we recall that $A(\,\_\,)$ means the conjugation by the exponential function; i.e., $A(g) \coloneqq \ln \circ g \circ \exp$, for any given function $g$, and we also recall that by $x\mapsto\exp^{n}(x)$ we mean the $n$-fold composition of the exponential.

\begin{proof}[Proof of Theorem A]
First of all, since in Lemma \ref{Lemma:NC group struc} we just proved that $\mathcal{NC}$ forms a group under composition, we are only looking at an element $g$ of:
\begin{equation*}
    \Aff \circ \mathcal{NC} \circ \cdots \circ A^{n}\mathcal{NC} \circ \cdots \circ \mathcal{NC} \circ \Aff.
\end{equation*}
We proceed by induction on $n$, being the base case $n=0$ obvious. Thus, let $n>0$ be any positive integer, let us take arbitrary affine elements $a, b \in \Aff$, and let $g$ be any element in:
\begin{equation*}
    a \circ \mathcal{NC} \circ \cdots \circ A^{n} \mathcal{NC}  \circ \cdots \circ \mathcal{NC} \circ b\,.
\end{equation*}
We stress that $g$ has a fixed-points free neighbourhood around zero if and only if $b \circ g \circ b^{-1}$ does. Hence, we may assume $b \equiv \id$. Consequently, we note then that the element $g$ can be rewritten as the sum of the affine term $a$ plus some exponentially small terms. But, if $a \not\equiv \id$, then that keeps $g$ away from the identity. Therefore, we may also assume $g$ to be in:
\begin{equation*}
    \mathcal{NC} \circ \cdots \circ A^{n}\mathcal{NC}  \circ \cdots \circ \mathcal{NC}.
\end{equation*}
Again by conjugation we may assume that there exists some element $g_{1} \in \mathcal{NC}$ such that $g$ is in:
\begin{equation*}
    g_{1} \circ A\mathcal{NC}\circ \cdots \circ A^{n}\mathcal{NC} \circ \cdots \circ A\mathcal{NC}.
\end{equation*}
But, once again we notice then that $g$ turns out to be the sum of $g_{1}$ and double exponentially small terms by Taylor expansion. So, by the Phragm\'en--Lindel\"{o}f arguments in Lemma \ref{LemPLSimpCochain}, if the Dulac series for $g_{1}$ is not the identity, it provides a neighbourhood without fixed points, else $g_{1} \equiv \id$ and then $g$ is in:
\begin{equation*}
    A\mathcal{NC}\circ \cdots \circ A^{n}\mathcal{NC} \circ \cdots \circ A\mathcal{NC}.
\end{equation*}
Then $g$ has a fixed-points free neighbourhood if and only if $A^{-1}(g)$ does. Nevertheless, $A^{-1}(g)$ is in:
\begin{equation*}
    \mathcal{NC}\circ \cdots \circ A^{n-1}\mathcal{NC} \circ \cdots \circ \mathcal{NC}.
\end{equation*}
So by induction we are done.
\end{proof}

\bibliographystyle{amsplain}
\bibliography{mybib}

\end{document}